\documentclass[journal]{IEEEtran}
\usepackage{amssymb}
\usepackage[cmex10]{amsmath}
\usepackage{amsthm}
\usepackage{color}
\usepackage[titlenumbered,ruled]{algorithm2e}
\usepackage{dsfont}
\usepackage{balance}
\usepackage{graphicx}
\usepackage{subcaption}
\usepackage{tikz}

\DeclareMathOperator*{\argmin}{\arg\!\min}

\newtheorem{definition}{Definition}
\newtheorem{theorem}{Theorem}
\newtheorem{proposition}{Proposition}

\newcommand{\E}{\mathcal{E}}

\newcommand{\Ne}{\ensuremath{\mathcal N}}
\newcommand{\Dk}{\ensuremath{D_{\kappa}}}

\newcommand{\tr}[1]{\ensuremath{\textbf{tr}\left(#1\right)}}
\newcommand{\expec}[1]{\ensuremath{\textbf{E}\left[#1\right]}}
\newcommand{\Sc}{\ensuremath{\mathcal S}}

\newcommand{\Sg}{\ensuremath{S}}
\newcommand{\Hopt}{\ensuremath{\hat{H}}}

\begin{document}
%
\title{Submodular Optimization for Consensus Networks with Noise-Corrupted Leaders}

\author{Erika Mackin and Stacy Patterson

\thanks{*This work was supported in part by NSF grants CNS-1527287 and CNS-1553340.}

\thanks{Erika Mackin and Stacy Patterson are with Department of Computer Science, Rensselaer Polytechnic Institute, Troy, New York 12180, USA. Phone:  518-276-2054. Fax: 518-276-4033. Email:
{\tt\small mackie2@rpi.edu, sep@cs.rpi.edu}}
}

\maketitle

\begin{abstract}
We consider the leader selection problem in a network with consensus dynamics where both  leader and follower agents are subject to stochastic external disturbances.  The performance of the system is quantified by the total steady-state variance of the node states, and the goal is to identify the set of leaders that minimizes this variance.
We first show that this performance measure can be expressed as a submodular set function over the nodes in the network.  We then use this result to analyze the performance of two greedy, polynomial-time algorithms for leader selection, showing that the leader sets produced by the greedy algorithms are within provable bounds of optimal.
\end{abstract}



\section{Introduction}

\IEEEPARstart{C}{onsensus} algorithms play an important role in networked systems in applications such as sensor fusion~\cite{xiao2005scheme}, autonomous formation control~\cite{BH06}, clock synchronization~\cite{EKPS04}, and distributed localization~\cite{BH09}.  Due to their importance, much study has been done on the performance of these algorithms, including their convergence rates and their robustness to external disturbances.

An important class of consensus algorithms is leader-follower consensus algorithms.  In leader-follower consensus,  a subset of nodes are \emph{leaders} that dictate the desired state of the network.  The remaining nodes are \emph{followers} that update their states using consensus dynamics.
The performance of a leader-follower consensus system depends both on the topology of the network and the locations of the leader nodes.
A natural question that arises is where to locate the leaders (i.e., at what nodes in the network?) so as to optimize some performance measure.
This problem is known as the \emph{leader selection problem}.

The optimal leader set of a given size $k$ can be found through an exhaustive search of all possible subsets of nodes of size $k$. This approach, however, is not computationally tractable for anything other than small networks or small values of $k$.  Therefore, much research has been done into deriving efficient approximation algorithms and corresponding bounds for leader selection.  Several works have investigated the leader selection problem using convergence rate as a performance measure~\cite{PMB08,RJME09,PS14,CABP14, MP17,YZSC17}. Most notably, the works~\cite{CABP12,CABP14} propose relaxations of the leader selection problem that admit efficient solutions.  While these relaxed formulations perform well in evaluations, there are no guarantees on the optimality of their solutions.

Another performance measure used for leader-follower systems is \emph{network coherence}, which is quantified by the total steady-state variance of the node states.
In systems with \emph{noise-free leaders}, follower nodes are subject to stochastic disturbances, while leader nodes are not.  Heuristic approaches have been proposed to optimize coherence in this setting for general undirected graphs~\cite{PB10,LFJ14,CBP14}.  
Further, analytic solutions have been developed for undirected paths and cycles for an arbitrary number of leaders~\cite{FL16,PMD16,P17}, and for regular trees when the number of leaders is restricted to two~\cite{PYZ17}.  A closed-form expression for coherence in Koch networks with a single noise-free leader was derived in \cite{YZSC17}, and
analysis of the asymptotic scaling of coherence in  1D and 2D directed lattice graphs when leaders are on the boundary was given in~\cite{LFJ12}. Of particular note is the work of Clark et al.~\cite{CBP14}, which showed that the leader selection problem with noise-free leaders can be expressed as an optimization problem over a submodular set function.  As such, a greedy polynomial-time algorithm can find a leader set with performance that is within a provable bound of optimal~\cite{Nemhauser1978}.

Several works have also considered the leader selection problem with \emph{noise-corrupted leaders}, i.e., networks in which both the leader and follower nodes are subject to stochastic disturbances. The work by Fitch and Leonard~\cite{FL16} developed expressions for determining the optimal two noise-corrupted leaders in undirected cycles and paths, while \cite{PYZ17} developed a closed-form expression for the coherence in cycles with two noise-corrupted leaders. The work by Lin et al.~\cite{LFJ14} gives lower and upper bounds on network coherence for an undirected network with noise-corrupted leaders. While this work proposes a similar polynomial-time algorithm to~\cite{PMB08,CBP14}, as well as to the one we study in this work, it does not give any bounds on the quality of the solutions produced by this algorithm. Finally, the recent work by Dhingra et al.~\cite{DCJ16} presents an algorithm for noise-corrupted leader selection in directed networks; however, no guarantee on the optimality of the solution is provided.

In this work, we consider the leader selection problem in undirected graphs with noise-corrupted leaders. We quantify the performance of the system, for a given set of leaders, by the network coherence.
We first show that this performance measure can be recast as a submodular set function over the set of possible leader nodes. We then use this result to
prove optimality bounds for greedy leader selection algorithms similar to those described in~\cite{LFJ14}.  
As far as we are aware, this note is the first work to give optimality bounds on algorithms for leader selection in networks with noise-corrupted leaders.
We note our proof technique was inspired by a  result in~\cite{SSLD15} on optimizing network coherence by adding edges to the network.
The proof in this work was later found to be flawed~\cite{SCL, O17}. We show that, in the problem considered in this work, this proof technique is valid.  Details are given in Section~\ref{analysis.sec}.

The remainder of this note is structured as follows. In Section~\ref{prelim.sec}, we describe our system model and problem formulation, as well as background on submodular functions. Section~\ref{analysis.sec} gives our analysis of the submodularity of the leader selection performance measure.  Section~\ref{algorithms.sec} describes the greedy algorithms for leader selection and presents analysis of the performance of these algorithms, followed by a brief numerical example demonstrating their performance in Section~\ref{numex.sec}. Finally, we conclude in Section~\ref{conclusion.sec}.

\section{Preliminaries} \label{prelim.sec}

\subsection{System Model}
We consider a network of $n$ nodes, modeled by a connected, undirected graph $G = (V,E)$. Each node $i$ has a scalar-valued state $x_i$.
Nodes are either followers or leaders.
The state of a leader is updated as,
\[
\dot{x}_i = - \sum_{j \in \Ne_i} (x_i - x_j)~-~\kappa_i x_i~+~w_i,
\]
where $w_i$ is a zero-mean, white stochastic disturbance and $\kappa_i$ is a real, positive number.
The state of a follower is updated as,
\[
\dot{x}_j = - \sum_{j \in \Ne_j} (x_j - x_k)~+~w_j,
\]
where $w_j$ is again a zero-mean, white stochastic disturbance.

Let $S$ denote the set of leaders.
The dynamics of the system can be written as,
\[
\dot{x} = - (L + \Dk D_S ) x + w,
\]
where $L$ is the Laplacian matrix of the graph, $\Dk$ is a diagonal matrix, with diagonal entries $\kappa_i$, and $D_S$ is a diagonal matrix where the $(i,i)^{th}$ component is 1 if node $i$ is a leader; the $(i,i)^{th}$ component is 0 otherwise.
For simplicity of notation we use $Q_S$ to denote $L + \Dk D_S$.
If the set $S$ is a singleton consisting of the node $v$, we simply write $Q_v$.
Provided $S \neq \emptyset$, the matrix $Q_S$ is positive definite~\cite{RJME09}. 

We note that $Q_S$ can also be interpreted as a grounded Laplacian matrix of a graph $\bar{G}$, which is defined as follows.  Given the graph $G$ and the leader set $S$, $\bar{G}$ is formed by adding a single node $\bar{s}$ to $G$ and adding an edge from each node $i \in S$ to s,  with edge weight $\kappa_i$. All other edges have weight 1.
Let $\bar{L}$ be the corresponding Laplacian.  By removing the $s^{th}$ row and column from $\bar{L}$, we obtain $Q_S$. 
. 

As in~\cite{PB10,LFJ14,CBP14}, we quantify the performance of the system, for a given set of leaders $S$, by the total steady-state variance of $x$, 
\[
H(S) := \lim_{t \rightarrow \infty} \sum_{i=1}^n \expec{x_i^2}.
\]
It is straightforward to show that, for $S \neq \emptyset$~\cite{LFJ14}:
\[
H(S) = \frac{1}{2} \tr{Q_S^{-1}}.
\]
The \emph{$k$-leader selection problem for noise-corrupted leaders} is: given a budget $k$,  identify a set of at most $k$ leaders that minimizes the total steady-state variance, i.e.,
\begin{equation}
\begin{array}{ll}
\text{minimize} & H(S) = \frac{1}{2}\tr{Q_S^{-1}} \\
\text{subject to} & |S| \leq k.
\end{array} \label{Hprob.eq} \tag{LS}
\end{equation}
We denote the minimal value of $H(S)$, over all possible leader sets $S$ with $|S| \leq k$, by $\Hopt$.

A recent work~\cite{LFJ14} proposed an efficient approximation algorithm for (\ref{Hprob.eq}), but did not provide bounds on the performance of solutions
generated by this algorithm. In the remainder of this work, we show that $H(S)$ can be related to a submodular set function and, using this relationship,
we derive bounds on the solutions produced by greedy algorithms similar to the ones proposed in~\cite{LFJ14}.

\subsection{Background}
Our analysis of the leader selection problem is based on theory related to submodular set functions, which are defined as follows.
\begin{definition}[\cite{Nemhauser1978}]
A  function $f:2^V \mapsto\mathbf{R}$, where $V$ is a finite set, is called \emph{submodular} if, for all $A, B \subseteq V$,
\[ 
f(A) + f(B) \geq f(A \cup B) - f(A \cap B).
\]
\end{definition}
Informally, a submodular function exhibits a ``diminishing returns'' property: the incremental benefit of adding an element to a set $S$ is more than the incremental benefit of adding that same element to a superset of $S$.

We  make use of the following definition in our analysis.
\begin{definition}
A set function $f:2^V \mapsto \mathbf{R}$ is called \emph{non-increasing} if for all  $A,B \subseteq V$,
if $A \subseteq B$, then $f(A) \geq f(B)$. The function $f$ is called \emph{non-decreasing} if for all $A, B \subseteq V$,
if $A \subseteq B$, then $f(A) \leq f(B)$.
\end{definition}

We also make use of the following theorem.
\begin{theorem}[\cite{L83}]
\label{submod}
A function $f:2^V \mapsto \mathbf{R}$ is submodular if and only if the derived set functions $f_a:2^{V-\{a\}} \mapsto \mathbf{R}$, defined by,
\begin{equation}
f_a(X) = f(X \cup \{a\})-f(X), \label{submod.eq}
\end{equation}
are non-increasing for all $a \in V$.
\end{theorem}

We note that maximizing a submodular function is an NP-hard problem~\cite{Nemhauser1978}.

\section{Leader Selection and Submodularity} \label{analysis.sec}

We introduce a  function $f(S)$ over  sets of leader nodes, such that maximizing $f(S)$ is equivalent to minimizing $H(S)$. We then prove that this set function $f$ is non-decreasing and submodular, properties which guarantee that leader sets generated by polynomial-time greedy algorithms are within a provable bound of optimal. These algorithms and bounds are presented in Section~\ref{algorithms.sec}. 

The function $f: 2^V \mapsto \mathbf{R}$ is defined as follows:
\begin{equation} \label{fdef.eq}
 f(S) = \left\{ \begin{array}{ll} 
0 & \text{if}~S = \emptyset \\
 C - \tr{ Q_S^{-1}} & \text{otherwise},
\end{array} \right.
\end{equation}
where $C = 2 \left(\max_{ v \in V}  \tr{Q_v^{-1}}\right)$.
Note that a set $\hat{S}$ that maximizes $f$ also minimizes $H$.

\begin{proposition}\label{fnd.prop}
The function $f$ is a non-decreasing function.
\end{proposition}
\begin{IEEEproof}
Let $\Sc_1 \subseteq \Sc_2 \subseteq V$.
First, we consider the case where $\Sc_1 = \Sc_2 = \emptyset$.  Then, $f(\Sc_1) = f(\Sc_2) = 0$,
which implies $f(\Sc_1) \leq f(\Sc_2)$.

Next, we consider $\Sc_1 = \emptyset$ and $\Sc_2 \neq \emptyset$.
We now show $f(\Sc_2) \geq 0$, which implies $f(\Sc_1) \leq f(\Sc_2)$. By definition,
\begin{align*}
f(\Sc_2) &= 2 \max_{v\in V} \tr{Q_v^{-1}} - \tr{Q_{\Sc_2}^{-1}}\\
&\geq 2 \max_{v \in \Sc_2} \tr{Q_v^{-1}} - \tr{Q_{\Sc_2}^{-1}}.
\end{align*}
If $|\Sc_2| = 1$, then $f(\Sc_2) \geq 0$ holds trivially. 
Otherwise, let $u = \arg\max_{v \in \Sc_2} \tr{Q_v^{-1}}$ and let $Z= \Sc_2 \setminus \{u\}$. It follows that:
\begin{align*}
Q_{\Sc_2} &= L + \Dk D_{\Sc_2} \\
&= L + \Dk D_Z + \Dk D_u.
\end{align*}
Since $L + \Dk D_Z$ is positive definite and $\Dk D_u$ is positive semidefinite, by Weyl's theorem, 
\[
\lambda_i\left(Q_{u}\right) \leq \lambda_i \left(Q_{Z} + \Dk D_u\right) = \lambda_i \left(Q_{\Sc_2}\right),
\]
for $i=1 \ldots n$.
This implies that $\tr{Q_u^{-1}} \geq \tr{Q_{\Sc_2}^{-1}}$, and thus, $f(\Sc_2) \geq 0$.

Finally, we consider the case where $\Sc_1 \neq \emptyset$. Then,
\begin{align*}
f(\Sc_1) - f(\Sc_2) &= C- \tr{Q_{\Sc_1}^{-1}} - \left( C- \tr{Q_{\Sc_2}^{-1}}\right) \\
&= \tr{Q_{\Sc_2}^{-1}} - \tr{Q_{\Sc_1}^{-1}}.
\end{align*}
Let $Z = \Sc_2 \setminus \Sc_1$ and note that $Q_{\Sc_2} = Q_{\Sc_1} + \Dk D_Z$.
By a similar application of Weyl's theorem as above, it holds that $\tr{Q_{\Sc_2}^{-1}} \leq \tr{Q_{\Sc_1}^{-1}}$,
or equivalently, that $f(\Sc_1) \leq f(\Sc_2)$.
\end{IEEEproof}

\begin{theorem} \label{submod.thm}
The set function $f$, defined in (\ref{fdef.eq}),  is submodular. 
\end{theorem}
\begin{IEEEproof}
To prove this theorem, we first define the set function ${f_a : 2^{V \setminus \{a\} } \mapsto \mathbf{R}}$,
\begin{equation}
f_a(\Sc) =  f( \Sc \cup \{a\}) - f(\Sc), \label{derived.eq}
\end{equation}
and show that it is monotone decreasing.

Let $\Sc_1 \subseteq \Sc_2 \subseteq V \setminus \{a\}$. 
We first consider the case where $\Sc_1 \neq \emptyset$ and $\Sc_2 \neq \emptyset$.
The proof of this case follows the same structure used in~\cite{SSLD15}. 
In this case, (\ref{derived.eq}) is equivalent to:
\begin{align*}
f_a(\Sc) &= C - \tr{Q^{-1}_{\Sc \cup \{a\}}} - \left(C -  \tr{Q^{-1}_{\Sc}}\right)\\
&=- \tr{(Q_{\Sc} + \Dk D_{a})^{-1}}+ \tr{Q^{-1}_{\Sc}},
\end{align*}
where $D_a$ denotes the diagonal matrix with $D_a(a,a) = 1$ and all other entries equal to 0.

We define the function $Q(t)$, for $t \in [0,1]$, as,
\begin{align*}
 Q(t) &= Q_{\Sc_1} + t(Q_{\Sc_2}-Q_{\Sc_1}). 
 \end{align*}
 Note that $Q(0) = Q_{\Sc_1}$ and $Q(1) = Q_{\Sc_2}$. Let
\begin{align}
\hat{f}_{a} (Q(t)) = -\tr{(Q(t)+ \Dk D_{a})^{-1}}+\tr{Q(t)^{-1}}. \label{fahat.eq}
\end{align}

We next take the derivative of $\hat{f}_a$ with respect to $t$,
\begin{align}
&\frac{d}{dt}\hat{f}_{a}(Q(t)) \nonumber  \\
& =  \frac{d}{dt}  \left( -\tr{ (Q(t)+ \Dk D_a)^{-1} }+\tr{Q(t)^{-1} } \right) \label{deriv1}\\
& =\tr{(Q(t)+ \Dk D_a)^{-1} (Q_{\Sc_2}-Q_{\Sc_1}) (Q(t)+ \Dk D_a)^{-1} } \nonumber \\
&~~~~~~~~~~-\tr{ Q(t)^{-1}(Q_{\Sc_2}-Q_{\Sc_1})Q(t)^{-1} } \label{deriv2} \\
&=\textbf{tr}\Big( \left[ (Q(t)+ \Dk D_a)^{-2}  -Q(t)^{-2}\right] (Q_{\Sc_2}-Q_{\Sc_1})\Big), \label{deriv3}
\end{align}
where (\ref{deriv2}) is obtained from (\ref{deriv1}) by applying the matrix derivative formula: 
\[
\frac{d}{dt} \tr{ Q(t)^{-1}} = -\tr{Q(t)^{-1}\frac{d}{dt}(Q(t))Q(t)^{-1}},
\]
and (\ref{deriv3}) is obtained from (\ref{deriv2}) by the cyclic property of the trace. 

We now show that (\ref{deriv3}) is non-positive. 
We define $d_a$ as the vector of all zeros except the $a^{th}$ component, which has value $\sqrt{\kappa_a}$.
Let 
\[X=\frac{Q(t)^{-1}\Dk D_a Q(t)^{-1}}{1 + d_a^T Q(t)^{-1} d_a}. \]
Using the Sherman-Morrison formula, we rewrite the first factor in (\ref{deriv3}) as: 
\begin{align}
&(Q(t)+ \Dk D_a)^{-2}  -Q(t)^{-2} = \Big( Q(t)^{-1} - X \Big)^2 - \Big(Q(t)^{-1}\Big)^2 \nonumber  \\
&~~~~~~~~~=-Q(t)^{-1}X - X Q(t)^{-1} + X^2 \nonumber \\
&~~~~~~~~~= - Q(t)^{-1}X  -X(Q(t)^{-1}-X)  \nonumber \\
&~~~~~~~~~= - Q(t)^{-1}X - X(Q(t)+ \Dk D_a)^{-1}. \label{inv2.eq}
\end{align}

Note that $Q(t)$ and $(Q(t) + \Dk D_a)$ are both grounded Laplacians, and therefore their inverses $Q(t)^{-1}$ and ${(Q(t) + \Dk D_a)^{-1}}$ are element-wise positive for all $t$ and $a$~\cite{M93}. 
From this, we can conclude that $X$ is also element-wise positive, since the numerator is clearly element-wise positive and the denominator is the positive scalar $1 + Q(t)^{-1}_{(a,a)}$. We can then see that $-Q(t)X$ and $(-X)(Q(t) + \Dk D_a)^{-1})$ are both element-wise negative matrices, and therefore, $(Q(t)+ \Dk D_a)^{-2}  -Q(t)^{-2}$ is also element-wise negative.

We also note that 
\begin{align*}
Q_{\Sc_2}-Q_{\Sc_1} &= \Dk D_{\Sc_2} - \Dk D_{\Sc_1},
\end{align*}
is a diagonal matrix where the $(i,i)^{th}$ component is 1 if ${i \in S_2 \setminus S_1}$ and $0$ otherwise.
Therefore,  
\begin{align}
\left[ (Q(t)+ \Dk D_a)^{-2}  -Q(t)^{-2}\right] (Q_{\Sc_2}-Q_{\Sc_1})  \label{negMat.eq}
\end{align}
is a matrix with columns that are either all zeros or correspond to the columns of $(Q(t)+ \Dk D_a)^{-2}  -Q(t)^{-2}$ with indices in $S_2 \setminus S_1$. The matrix (\ref{negMat.eq}) is therefore non-positive, as is its trace.  It follows that $\frac{d}{dt}\hat{f}_{a}(Q(t))$ is non-positive. 

Consider the following equality,
\begin{equation*}
\hat{f}_a(Q(1)) = \hat{f}_a(Q(0))+\int_0^1  \textstyle \frac{d}{dt}\hat{f}_a(Q(t))dt.
\end{equation*}
Since $\hat{f}_a(Q(1)) = f_{a}(\Sc_2)$ and $\hat{f}_a(Q(0)) = f_{a}(\Sc_1)$, and since, as shown above,  $\frac{d}{dt}\hat{f_a}(Q(t))dt$ is non-positive, 
we have $ f_{a}(\Sc_1) \geq f_{a}(\Sc_2)$.  Therefore $f_{a}$ is non-increasing over $\Sc \subseteq V$, $\Sc \neq \emptyset$.

If $\Sc_1 = \Sc_2 = \emptyset$, then  $f_a(\Sc_1) = f_a(\Sc_2)$, and so it also holds that $f_{a}$ is non-increasing over $S = \emptyset$.

Finally, we consider $\Sc_1 = \emptyset$ and $\Sc_2 \neq \emptyset$.
Then,
\begin{align*}
&f_a(\Sc_1) - f_a(\Sc_2) = f(\{a\}) - \left(f(\Sc_2 \cup \{a\}) - f(\Sc_2) \right)) \\
&= \left(C - \tr{Q_a^{-1}}\right) - \big(C - \tr{Q_{\Sc_2 \cup \{a\}}^{-1}} - \left(C - \tr{ Q^{-1}_{\Sc_2}} \right) \big) \\
&= \left(C - \tr{Q_a^{-1}}\right)  - \left(\tr{ Q^{-1}_{\Sc_2}} - \tr{Q_{\Sc_2 \cup \{a\}}^{-1}}\right). 
\end{align*}
With $C= 2 \max_{v \in V} \tr{Q_v^{-1}}$, we have,
\begin{align*}
 2 \max_{v \in V} \tr{Q_v^{-1}} - \tr{Q_a^{-1}} &\geq \max_{v \in V} \tr{Q_v^{-1}}  \\
&\geq \max_{u \in \Sc_2} \tr{Q_u^{-1}}.
\end{align*}
By Proposition~\ref{fnd.prop}, $\max_{u \in \Sc_2} \tr{Q_u^{-1}} \geq \tr{Q_{\Sc_2}^{-1}}$.
Therefore, $f_a(\Sc_1) - f_a(\Sc_2) \geq 0$.

Thus, $f_{a}$ is monotone decreasing over all subsets of $V$, and by Theorem~\ref{submod}, $f$ is submodular. 
\end{IEEEproof}

\subsection*{A note on the correctness of the proof method.}
A similar method to the above proof was first presented in~\cite{SSLD15}  to show that coherence, as a function of the set of edges that can be added to a noisy consensus network, can be captured by a submodular function. This same technique was later used to show the submodularity of functions in other network design problems such as: adding edges in networks with noisy consensus dynamics \cite{SSLD15}, adding edges in networks with noisy consensus dynamics with stubborn agents \cite{MP17}, and sensor and actuator placement for optimal controllability \cite{SCL16}. Unfortunately, the proofs in these works relied on a faulty assumption about the relationship between negative semidefinite matrices and their squares~\cite{SCL}.
And, in fact, it has since been shown that many of these set functions are not submodular~\cite{O17}.


As an illustration, we consider the problem of selecting edges to add to a noisy consensus network to optimize its coherence~\cite{SSLD15}. The system dynamics are: 
\begin{align*}
\dot{x}=-(L+L_\E)x+w,
\end{align*}
where $w$ is a vector of zero-mean, white noise processes, $\E$ is the set of edges added, and $L_\E$ is the corresponding Laplacian. The coherence of the network is:
 \[ H(\E) = \frac{1}{2}\tr{(L+L_{\E})^{\dagger}}, \]
 where $L^{\dagger}$ denotes the pseudo-inverse of the Laplacian,
 and the set function of interest is: 
 \[ f(\E) = \tr{L^{\dagger}}-\tr{(L+L_{\E})^{\dagger}}, \]
 corresponding to (\ref{fdef.eq}) in our problem setting.
 The functions $f_a$ and $\hat{f}_a$ are defined similarly to those in equations (\ref{derived.eq}) and (\ref{fahat.eq}).
The corresponding statement to  (\ref{deriv3}) is
\begin{align}
\textbf{tr}\Big( \left[ (Q(t)+ L_a)^{-2}  -Q(t)^{-2}\right] (L_{\E_2}-L_{\E_1})\Big) \leq 0, \label{oldTrace}
\end{align}
where  ${Q(t)=L +L_{\E_1} + t(L_{\E_2}-L_{\E_1}).}$

In the proof in~\cite{SSLD15}, the authors argue that because ${(Q(t)+ L_a)^{-1} -Q(t)^{-1}}$ is negative semidefinite, ${(Q(t)+ L_a)^{-2}  -Q(t)^{-2}}$ is negative semidefinite as well.
They use this assumption to show that the inequality (\ref{oldTrace}) holds.  As shown in~\cite{O17}, this relationship does not hold in general, nor does it hold in the problems in~\cite{MP17, SSLD15, SCL16}. In fact, this relationship does not necessarily hold in (\ref{deriv3}). Our proof does not rely on this assumption, but rather we base our argument on the element-wise positivity and negativity of the two factors in (\ref{deriv3}).

\section{Algorithms} \label{algorithms.sec}
In this section, we describe and analyze two greedy algorithms for the $k$-leader selection problem with noise-corrupted leaders.
These algorithms are problem-specific variations of algorithms presented in~\cite{Nemhauser1978} for maximizing a general 
submodular function.

\subsection{Algorithm Descriptions}
The first algorithm, the Greedy Algorithm, is given in Algorithm~\ref{greedy.alg}.
The algorithm is initialized with an empty leader set.  It first selects the single node $v$ that minimizes $\tr{Q_v^{-1}}$ and adds that to the leader set $S$.
Then, in each iteration, it selects the node $v$ that gives the smallest value of 
$\tr{Q_{\Sg \cup \{v\}}^{-1}}$ and adds that node to $S$. This continues until either $|S|= k$ or no node whose addition would further decrease $H(S)$ is found.

The second algorithm, the Swap Algorithm, is given in Algorithm~\ref{swap.alg}.
The Swap Algorithm takes an arbitrarily-selected leader set $S$ of size $k$ as input. 
To improve the performance of the leader set, a possible swap is looked for by repeatedly exchanging a single leader with a single follower. Call this potential leader set $S'$. 
If $\tr{Q_{S'}^{-1}} < \tr{Q_{S}^{-1}}$, then $S'$ becomes the new leader set. This process is repeated until no possible swap is found that decreases the value of $\tr{Q_{S}^{-1}}$.

Algorithm \ref{greedy.alg} is equivalent to the One-Leader-at-a-Time Algorithm in \cite{LFJ14}. 
Algorithm \ref{swap.alg} is similar to the Swap Algorithm in \cite{LFJ14}, with one difference.  In~\cite{LFJ14}, the input to the Swap Algorithm is the output from the One-Leader-at-a-Time Algorithm rather than an arbitrary leader set as in Algorithm~\ref{swap.alg}.  

As shown in~\cite{LFJ14}, the Greedy Algorithm requires $O(n^3)$ operations.
Further, \cite{LFJ14} shows that, in the Swap Algorithm, to evaluate the benefit of a potential swap requires $O(n)$ operations, and, once a beneficial swap has been identified, $O(n^2)$ operations are required to compute the new $Q_S^{-1}$.  
It has been shown that, for some submodular functions, such a swap algorithm may take an exponential number of iterations~\cite{Nemhauser1978}. It is an open question whether this worst-case time complexity holds for Algorithm~\ref{swap.alg}.

\begin{algorithm}[t]
\SetKwInOut{Input}{Input}
\SetKwInOut{Output}{Output}
\SetAlgoNoLine
\Input{$G=(V,E)$, $\kappa$, \text{maximum number of leaders $k$}}
\Output{Set of leader nodes $\Sg$} 
\KwSty{Initialize:} $\Sg \gets \emptyset$, $i \gets 0$ \\
$v \gets \argmin_{u \in V} \tr{Q_v^{-1}}$ \\
$\Sg \gets S \cup \{v\}$ \;
\For{$i=2 \ldots k$}{
	$v \gets \argmin_{u \in V\setminus \Sg} \tr{Q_{{\Sg} \cup \{u\}}^{-1}}$\\
	\eIf{$\tr{Q_{\Sg \cup \{v\}}^{-1}} < \tr{Q_\Sg^{-1}}$}{
		$\Sg \gets \Sg \cup \{v\}$
	}
	{
		return $\Sg$
	}
}
return $\Sg$
\caption{Greedy Algorithm for $k$-leader selection.} \label{greedy.alg}
\end{algorithm}

\begin{algorithm}[t] 
\SetKwInOut{Input}{Input}
\SetKwInOut{Output}{Output}
\SetAlgoNoLine
\Input{$G=(V,E)$, $\kappa$, \text{arbitrary leader set} $S$, $|S|=k$ }
\Output{Set of leader nodes $\Sg$} 
\KwSty{Initialize:} $T \gets V \setminus S$, $n \gets |V|$, $decreased \gets true$ \\
\While {decreased} {
	\textbf{label:}  \For{$i \in S, j \in T$}{
		\If{$\tr{Q_{S \cup \{j\} \setminus \{i\}}^{-1}} < \tr{Q_{S}^{-1}}$}{
			$S \gets S \cup \{j\} \setminus \{i\}$\\
			$T \gets T \cup \{i\} \setminus \{j\}$\\
			\textbf{break label}
		}
	}
	$decreased = false$
}
return $S$
\caption{Swap Algorithm for $k$-leader selection.} \label{swap.alg}
\end{algorithm}

\subsection{Performance Bounds}

We now present analysis of the performance of the leader selection algorithms.
\begin{theorem}
\label{greedybound.thm}
For a graph $G=(V,E)$ and number of leaders $k$, let $S_g$  be the leader set returned by Algorithm \ref{greedy.alg}, and let 
$\Hopt = \min_{S, |S| \leq k}H(S)$.  Then:
\begin{enumerate}
\item If $|S_g|<k$, then $H(S_g)=\Hopt$, and $S_g$ is an optimal leader set.
\item If $|S_g|=k$, then 
\begin{align}
H(S_g) \leq \left(1-\frac{1}{e}\right)\Hopt  + \frac{B}{e}, \label{greedybound.eq}
\end{align} 
where $B = \max_{ v \in V}  \tr{Q_v^{-1}}$.
\end{enumerate}
\end{theorem}

\begin{proof}
In Algorithm~\ref{greedy.alg}, the first leader that is selected is the node $v$ that minimizes $\tr{Q_v^{-1}}$.  This is equivalent to selecting the node $v = \arg \max_{u \in V} f(\{u\})$.
In each iteration of the algorithm, an additional leader is added to the set $S_g$ such that $H(S_g)$ is minimized, or equivalently, such that $f(S_g)$ is maximized. 
Algorithm~\ref{greedy.alg} is thus equivalent to a greedy algorithm that maximizes $f$.  
By Proposition \ref{fnd.prop} and Theorem \ref{submod.thm}, the function $f$ is non-decreasing and submodular, respectively.
Thus, by Proposition 4.2 in~\cite{Nemhauser1978}, if the greedy algorithm terminates with $|S_g| < k$, the set $S$ is optimal for $f$ and, therefore, also for $H$, 
which proves property 1.
By Theorem 4.2 in \cite{Nemhauser1978}, if the greedy algorithm terminates with $|S_g| = k$, then 
\begin{align}
f(S_g) \geq \left(1 - \left(\frac{k-1}{k}\right)^k\right)f(\hat{S}) \geq \left(1 - \frac{1}{e}\right)f(\hat{S}), \label{bound1}
\end{align}
where $\hat{S}$ is any leader set such that $\frac{1}{2}\tr{Q_{\hat{S}}^{-1}}=\Hopt$.
By applying the relationship between $f$ and $H$ to (\ref{bound1}), we obtain the bound (\ref{greedybound.eq}) in property 2.
\end{proof}
It has been shown that the bound (\ref{bound1}) is tight, meaning there is some submodular set function that achieves the worst case bound.
Futher, for a general submodular set function, this is the best achievable bound for \emph{any} polynomial-time algorithm, unless $P = NP$ \cite{CBP14}.  It remains an open question
whether this bound is tight for the specific function $f$ in  (\ref{fdef.eq}).

\begin{theorem}
\label{swapbound.thm}
For a graph $G=(V,E)$ and number of leaders $k$,  let $S_s$ be the leader set returned by Algorithm \ref{swap.alg},
 and let $\Hopt = \min_{S, |S| \leq k}H(S)$. Then,
\begin{align}
H(S_s) \leq \left(1-\frac{k-1}{2k-1}\right)  \Hopt + B\left(\frac{k-1}{2k-1}\right),
\label{swapbound.eq}
\end{align}
where $B=\max_{ v \in V}  \tr{Q_v^{-1}}$.
\end{theorem}
\begin{proof}
Following a similar argument to that in the proof of Theorem \ref{greedybound.thm}, it holds that the set $S_s$ produced by Algorithm~\ref{swap.alg} is equivalent to 
what would be produced by a swap algorithm that seeks to maximize $f$. Since $f$ is submodular and non-decreasing, by Theorem 5.1 in \cite{Nemhauser1978}, we have 
\begin{align}
f(S_s) \geq \left(1 - \frac{k-1}{2k-1}\right) f(\hat{S}), \label{swapfbound.eq}
\end{align}
where $\hat{S}$ is any leader set such that $\frac{1}{2}\tr{Q_{\hat{S}}^{-1}}=\Hopt$.
By applying the relationship between $f$ and $H$ to (\ref{swapfbound.eq}), we obtain the bound (\ref{swapbound.eq}).
\end{proof}

As with the greedy algorithm, the bound (\ref{swapfbound.eq}) is tight~\cite{Nemhauser1978}.
The question of whether this bound is tight for Algorithm \ref{swap.alg} remains open.

We note that, for general submodular functions, the bound (\ref{swapfbound.eq}) is worse than the bound (\ref{bound1}).   In other words, if one were to take the output of the greedy algorithm and use it as the input to the swap algorithm, as is done in~\cite{LFJ14}, there is no guarantee that the swap algorithm would improve upon the greedy solution. In fact, there are submodular functions for which it has been shown that this approach yields no improvement over the greedy algorithm alone~\cite{Nemhauser1978}.  However,~\cite{LFJ14} demonstrated that, in practice, the swap algorithm can lead to improvements on the greedy solution.

\section{Numerical Example} \label{numex.sec}
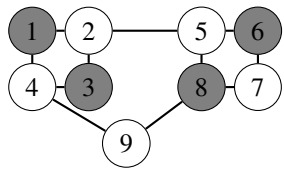
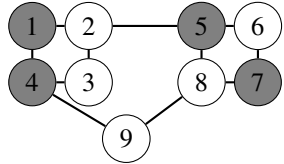
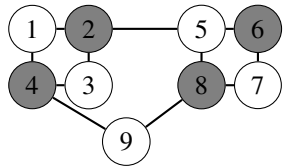
\begin{figure}

\begin{subfigure}[c]{\linewidth}
\centering
\begin{tikzpicture}[scale=0.5]

\node[shape=circle,draw=black, fill=gray] (1) at (0,1.5) {1};
\node[shape=circle,draw=black] (2) at (1.5,1.5) {2};
\node[shape=circle,draw=black, fill=gray] (3) at (1.5,0) {3};
\node[shape=circle,draw=black] (4) at (0,0) {4};
\node[shape=circle,draw=black] (5) at (4.5,1.5) {5};
\node[shape=circle,draw=black, fill=gray] (6) at (6,1.5) {6};
\node[shape=circle,draw=black] (7) at (6,0) {7};
\node[shape=circle,draw=black, fill=gray] (8) at (4.5,0) {8};
\node[shape=circle,draw=black] (9) at (2.5,-1.5) {9};

\foreach \from/\to in {1/2,2/3,3/4,4/1,5/6,6/7,7/8,8/5,8/9, 4/9,2/5}
    \draw[thick] (\from) -- (\to);
    
\end{tikzpicture} 
\caption{Optimal Leader Set.}
\end{subfigure}

\begin{subfigure}[c]{\linewidth}
\vspace{3mm}
\centering
\begin{tikzpicture}[scale=0.5]

\node[shape=circle,draw=black, fill=gray] (1) at (0,1.5) {1};
\node[shape=circle,draw=black] (2) at (1.5,1.5) {2};
\node[shape=circle,draw=black] (3) at (1.5,0) {3};
\node[shape=circle,draw=black, fill=gray] (4) at (0,0) {4};
\node[shape=circle,draw=black, fill=gray] (5) at (4.5,1.5) {5};
\node[shape=circle,draw=black] (6) at (6,1.5) {6};
\node[shape=circle,draw=black, fill=gray] (7) at (6,0) {7};
\node[shape=circle,draw=black] (8) at (4.5,0) {8};
\node[shape=circle,draw=black] (9) at (2.5,-1.5) {9};

\foreach \from/\to in {1/2,2/3,3/4,4/1,5/6,6/7,7/8,8/5,8/9, 4/9,2/5}
    \draw[thick] (\from) -- (\to);

\end{tikzpicture} 
\caption{Leaders selected by Greedy Algorithm.}

\end{subfigure}

\begin{subfigure}[c]{\linewidth}
\vspace{3mm}
\centering
\begin{tikzpicture}[scale=0.5]

\node[shape=circle,draw=black] (1) at (0,1.5) {1};
\node[shape=circle,draw=black, fill=gray] (2) at (1.5,1.5) {2};
\node[shape=circle,draw=black] (3) at (1.5,0) {3};
\node[shape=circle,draw=black, fill=gray] (4) at (0,0) {4};
\node[shape=circle,draw=black] (5) at (4.5,1.5) {5};
\node[shape=circle,draw=black, fill=gray] (6) at (6,1.5) {6};
\node[shape=circle,draw=black] (7) at (6,0) {7};
\node[shape=circle,draw=black, fill=gray] (8) at (4.5,0) {8};
\node[shape=circle,draw=black] (9) at (2.5,-1.5) {9};

\foreach \from/\to in {1/2,2/3,3/4,4/1,5/6,6/7,7/8,8/5,8/9, 4/9,2/5}
    \draw[thick] (\from) -- (\to);

\end{tikzpicture} 
\caption{Leaders selected by Swap Algorithm, with input $S=\{1,2,4,5\}$.}

\end{subfigure}
\caption{Leaders chosen for a graph $G$ by Algorithms~\ref{greedy.alg} and \ref{swap.alg}, with $k=4$. Leaders are shown in gray. }
\label{algEx.fig}
\end{figure}

For the graph $G$, shown in Fig. \ref{algEx.fig}, the greedy algorithm chooses as leaders, in order, nodes $5$, $4$, $7$, and $1$. The coherence of the network is then $H(S)=3.0910$.
For the swap algorithm, given the same $G$ and a randomly chosen starting leader set, $S=\{1,2,4,5\}$, the algorithm performs three rounds of swaps, terminates, and outputs the leader set $S=\{2,4,6,8\}$, with resulting network coherence $H(S)=3.0576$. The optimal leader set is $\hat{S}=\{1,3,6,8\}$, which is distinct from the output of both algorithms. The coherence of the network with optimal leaders is  $\hat{H} = 3$. 

When the output of the greedy algorithm is used as the initial set in the swap algorithm, then the algorithm terminates after two rounds of swaps and outputs $S=\{1,3,5,7\}$, with coherence $H(S)=3.0546$. In this case, the swap algorithm does improve upon the performance of the greedy algorithm, although the result is still sub-optimal.

\balance

\section{Conclusion} \label{conclusion.sec}

We have studied the $k$-leader selection problem in leader-follower consensus networks with noise-corrupted leaders. 
System performance is quantified by the network coherence, which is the  total steady-state variance of the nodes.
We first showed that the network coherence can be expressed as a submodular set function. 
Using this result,  we then derived bounds on the performance of two greedy  leader selection algorithms.
In future work, we plan to extend our analysis to leader-follower consensus in directed graphs. 



%
%
%



%
\bibliographystyle{IEEEtran}
\bibliography{grCons}

\end{document}